\documentclass[reqno]{amsart}
\usepackage{amsmath,amssymb,epsfig,graphicx}
\usepackage{hyperref}
\usepackage[hmarginratio=1:1, bottom=3cm]{geometry}

\usepackage[utf8]{inputenc} 
\usepackage{color}
\usepackage{enumerate}
\usepackage{bbm}
\usepackage{setspace}
\usepackage{subfig}
\usepackage{subfloat}
\usepackage{enumitem}
\usepackage{enumerate}
\usepackage{cancel}
\usepackage[normalem]{ulem}
\usepackage{booktabs}

\usepackage{multirow}

\newtheorem{remark}{Remark}
\newtheorem{prop}{Proposition}[section]

\newcommand{\cA}{{\mathcal A}}

\newcommand\R{{\mathbb R}}
\renewcommand\P{{\mathbb P}}

\newcommand\E{{\mathbb E}}

\newcommand\N{{\mathbb N}}
\newcommand\Z{{\mathbb Z}}
\newcommand\vare{{\varepsilon}}
\definecolor{newgreen}{rgb}{0,0.6,0.3}

\newcommand{\cE}{{\mathcal E}}

\newcommand{\dx}{\Delta x}

\newcommand{\veps}{\varepsilon}
\newcommand{\beqn}{\begin{equation}}
\newcommand{\eeqn}{  \end{equation}}
\newcommand{\beqno}{\begin{equation*}}
\newcommand{\eeqno}{  \end{equation*}}
\newcommand{\be}{\begin{eqnarray}}
\newcommand{\ee}{  \end{eqnarray}}
\newcommand{\beno}{\begin{eqnarray*}}
\newcommand{\eeno}{  \end{eqnarray*}}

\newcommand{\conv}{\rightarrow}
\numberwithin{equation}{section}
\newcommand{\VF}{V}

\usepackage{floatrow}

\begin{document}

\author[Athena Picarelli]{Athena Picarelli}
\address{Economics Department, University of Verona, Via Cantarane 24, 37129, Verona, Italy}
\email{athena.picarelli@univr.it}
\author[Christoph Reisinger]{Christoph Reisinger}
\address{Mathematical Institute, University of Oxford, Andrew Wiles Building, OX2 6GG, Oxford, UK}
\email{christoph.reisinger@maths.ox.ac.uk}

\title[Probabilistic error estimates for optimal control problems]{Probabilistic error analysis
 for some approximation schemes to optimal control problems}
\maketitle

\begin{abstract}
{We introduce a class of numerical schemes for  optimal stochastic control problems based on a novel Markov chain approximation, which uses, in turn, a piecewise constant policy approximation, Euler-Maruyama time stepping, and
a Gau{\ss}-Hermite approximation of the Gau{\ss}ian increments.
We provide lower error bounds of order arbitrarily close to 1/2 in time and 1/3 in space for Lipschitz viscosity solutions,
coupling probabilistic arguments with regularization techniques as introduced by Krylov. {The corresponding order of the upper bounds is 1/4 in time and 1/5 in space.
For sufficiently regular solutions, the order is 1 in both time and space for both bounds. Finally, we propose techniques for further improving the accuracy of the individual components of the approximation.}
}
\end{abstract}

\noindent
{\footnotesize {\bf Keywords}: Optimal stochastic control, Markov chain approximation schemes, error estimates}


\section{Introduction}\label{sect:intro}
Let $\left(\Omega, \mathbb F, \P\right)$ be a probability space with filtration  $\left\{\mathbb F_t, t\geq0\right\}$ induced by an
 $\R^p$-Brownian motion $B$ for some $p\geq 1$.
We consider a controlled process governed by
\begin{equation}\label{SDE}
\left\{\begin{array}{l}
\mathrm{d}X_s=\mu(s,X_s,\alpha_s)\,\mathrm{d}s+ \sigma(s,X_s,\alpha_s) \,\mathrm{d} B_s, \qquad s\in(t, T), \\
X_t=x,\end{array}\right.
\end{equation}
where $\mu$ and $\sigma$ take values, respectively, in $\R^d$ and $\R^{d\times p}$.
We assume that the control vector process $\alpha$ {belongs to the set  $\cA$ of progressively measurable processes with values in $A\subseteq \R^q$}.
For any $x\in\R^d$, we will denote by $X^{t,x,\alpha}_{\cdot}$ the {unique} strong solution of \eqref{SDE},
{under the assumptions specified later}. {To simplify the notation, where no ambiguities arise, we will indicate the starting point $(t,x)$ of the processes involved as a subscript in the expectation, i.e. $\E_{t,x}[\cdot]$.}

Given {$T>0$ and two real valued functions  $g$ and $\psi$, namely the running and terminal cost, respectively,}   for any $x\in\R^d$,
$t\in [0,T]$, the value function of the optimal control problem we consider is defined by 
\begin{align}
\label{eq:value_v}
v(t,x):=\sup_{\alpha\in\cA}\E_{t,x}\left[\psi(X^{\alpha}_{T}) + \int^T_t g(s,X^\alpha_s,\alpha_s)\mathrm d s\right].
\end{align}


It is well known that this problem is related to the solution of a second order  Hamilton-Jacobi-Bellman (HJB) equation for which, in the general case, solutions are considered 
the viscosity sense (see, for instance, \cite{CIL92}). 
Furthermore, explicit solutions for this kind of nonlinear equations are rarely available, so that their numerical approximation becomes vital.
The seminal work by Barles and Souganidis \cite{BS91} establishes 
the basic framework for convergence of numerical schemes to viscosity solutions of HJB equations. The fundamental properties required are: monotonicity, consistency, and stability of the scheme.
We recall that, in multiple dimensions, standard finite difference schemes are in general non-monotone.
As an alternative to finite difference schemes,
semi-Lagrangian (SL) schemes \cite{M89,CF95,DJ12} 
are monotone by construction.
The schemes we introduce belong to this family.

In general, the provable order of convergence for second order HJB equations is significantly less than one.
By a technique pioneered by Krylov based on ``shaking the coefficients'' and mollification to construct smooth sub- and/or super-solutions,
\cite{K97, K00, BJ02, BJ05, BJ07} prove certain fractional convergence orders, mainly using PDE-based techniques
{
which rely on a comparison principle between viscosity sub- and super-solutions and estimates on the consistency error of the numerical scheme}.

{
Here, we study a new family of SL schemes based on a discrete time approximation of the optimal control problem.
We use purely probabilistic techniques and a direct comparison of the two optimal control problems to obtain error estimates which are,
to our knowledge, the best ones available in the literature under these weak assumptions.}

An important step in order to define our scheme is to approximate the set of controls $\cA$ by piecewise constant controls. This introduces an asymmetry between the upper and the lower bound of the error, {and it is the lower bound where we get an improvement over known results}.

The approach most closely related to ours is arguably \cite{Kry99}, especially Section 5 therein, where approximations based on piecewise constant policies and subsequently on discrete-time random walks are studied. The analysis there utilizes a combination of stochastic and analytic techniques, in particular through controlling the approximation error by the truncation error between the generator of the controlled process and its discrete approximation, and It{\^o}'s lemma with the dynamic programming principle to aggregate the local error over time.
We will be able to improve the order of the error bounds partly by using recent improved estimates for the piecewise constant policy approximation in \cite{JakoPicaReisi18}, but also
by avoiding the use of the truncation error, the order of which is limited to 1, and replacing it by a direct estimate of the strong and weak approximation error of the scheme for the stochastic differential equation.

The main contributions of this paper are as follows.
\begin{itemize}
\item
We propose
new discrete approximations of controlled diffusion processes based on piecewise constant controls over intervals of length $h$ and $M$ Gau{\ss}-Hermite points.
\item
We present a novel analysis technique for the resulting 
semi-discrete approximations by purely probabilistic arguments and direct use of the dynamic programming principle.
\end{itemize}
This allows us to derive
one-sided, lower error bounds of order $h^{(M-1)/2M} + \Delta x^{(M-1)/(3M-1)}$ for timestep $h$ and spatial mesh size $\Delta x$, for Lipschitz viscosity solutions (assumptions {(H1)} to {(H3)} {below}). They coincide with the two-sided bounds in \cite{DJ12} for the standard linear-interpolation SL scheme, i.e.\ $M=2$, and improve them for $M>2$.
The achieved upper bounds are identical to \cite{DJ12}, {i.e.\ of order $h^{1/4} + \Delta x^{1/5}$}.
For sufficiently smooth solutions, the corresponding error bounds are of order $1$ in both $h$ and $\Delta x$.

The paper is organised as follows. In Section \ref{sec:setting} we present the setting and the main assumptions for the optimal control problem and we describe 
 the piecewise constant policy approximation.
In Section \ref{sect:disc_scheme} the Markov-chain approximation scheme is introduced and error bounds are obtained  in Section \ref{sec:lower}.
Section \ref{sec:smooth_case} discusses the order obtained in the case of smooth solutions and further improvements of the components of the scheme, including higher order time stepping and interpolation, while { Section \ref{sec:num} demonstrates the improvement achieved by a higher order scheme numerically.}
Section \ref{sec:concl} concludes.


\section{Main assumptions and preliminaries}\label{sec:setting} 
 {In what follows, $|\cdot|$ denotes the Euclidean norm in $\R^n$ for any $n\geq 1$ and $\|\cdot\|$ its induced matrix norm. }  
We consider standard assumptions on the optimal control problem:
\begin{itemize}
\item[\bf{(H1)}]
 $A$  {is a compact subset of a separable metric space};
 \item[\bf{(H2)}] $\mu:[0,T]\times \R^d \times A\to\R^d$ and $\sigma:[0,T]\times \R^d\times A\to \R^{d\times p}$ are continuous functions and there exists $C_0\geq 0$ such that for any  $t, s \in [0,T], x, y \in\R^d, a\in A$
 \begin{eqnarray*}
&& \left|\mu(t,x,a)-\mu(s,y,a)\right|+\left\|\sigma(t,x,a)-\sigma(s,y,a)\right\|\leq C_0  \left(\left|x-y\right|+ \left|t-s\right|^{1/2}\right); 
 \end{eqnarray*}
\item[\bf{(H3)}] $\psi:\R^d\to\R$ and $g:[0,T]\times \R^d \times A\to\R$  are  continuous functions and there exists  $L\geq 0$ such that for any  $t, s \in [0,T], x, y \in\R^d, a\in A$
$$
\left|\psi(x) - \psi(y)\right|+ \left|g(t,x,a)-g(s,y,a)\right|\leq L \left(\left|x-y\right|+ \left|t-s\right|^{1/2}\right).
$$
\end{itemize}

Under these assumptions one can prove the following regularity result on  $v$:
\begin{prop}[{\cite[Proposition 3.1, Chapter IV]{YZ}}]\label{prop:vLip}
Let (H1)-(H3) be satisfied. There exists $C\geq 0$  such that for any $x,y\in\R^d$ and $t\in [0,T]$
$$
\left|v(t,x)-v(t,y)\right|\leq L C|x-y|
$$
 (where $L$ denotes the Lipschitz constant of $\psi$ and {$C$ only depends on $T$ and $C_0$ in assumption (H2)}).
\end{prop}
{ Hereafter we assume $g\equiv 0$. Indeed, if this is not the case it is possible to consider the augmented dynamics $(X_\cdot, Y_\cdot)\in \R^{d+1}$ with $\mathrm d Y_s =  g(s,X_s,\alpha_s)\mathrm d s$ and the modified terminal cost $\varphi(x,y) := \psi(x)+y$.  Denoting by $w(t,x,y)= \sup_{\alpha\in \mathcal A} \mathbb E_{t,x,y}\left[ \varphi(X^\alpha_T,Y^\alpha_T)\right]$, it is sufficient to observe that $v(t,x) = w(t,x,0)$  to recover the aforementioned case.\\  
A fundamental property satisfied by the value function $v$ is the following  Dynamic Programming Principle (DPP) {(see, for instace \cite[Theorem 3.3, Chapter IV]{YZ}): for any $0\leq h\leq T-t$},  one has }
\begin{equation}\label{eq:DPP_v}
v(t,x)= \underset{\alpha\in \cA}\sup \;\E_{t,x}\left[v(t+h,X^{\alpha}_{t+h})\right].
\end{equation}


{The main ideas of our approach apply to a general class of discrete-time schemes.} 
Let {$N\in \mathbb N^*$ and $h=T/N>0$}. We introduce a time mesh 
$
 t_n=nh,
$
for $n=0,\dots,N$.

The first step in our approximation is to introduce a time discretization of the control set. We consider  the set $\cA_h$ of controls in $\cA$ that are constant in each interval $[t_n, t_{n+1})$ for $n=0\dots N-1$, i.e.
$$
{
\cA{_h}:=\Big\{\alpha\in\mathcal A:  
{\; \forall \omega \in \Omega } \;  \exists a_i\in A, \, i=0,\ldots, N-1,  \;  \text{ s.t. }  \alpha_s{(\omega)} \equiv \sum^{N-1}_{i=0} a_i \mathbbm 1_{s\in [t_i,t_{i+1})} 
\Big\}.
}
$${
In what follows, we will identify any element of $\alpha\in \mathcal A_h$ by the sequence of random variables $a_i$ taking values in $A$ (denoted  by $a_i \in A$ for simplicity) and will write $\alpha\equiv (a_0,\ldots, a_{N-1})$.}
We denote by $v_h$ the value function obtained by restricting the supremum in \eqref{eq:value_v} to controls in $\cA_h$, that is 
\begin{equation}\label{eq:value_v_delta}
v_h(t,x):=\underset{\alpha\in\cA_h}\sup \E_{t,x}\left[\psi(X^{\alpha}_T)\right].
\end{equation}
Clearly, since $\cA_h\subseteq \cA$, one has for any $t\in[0,T]$, $x\in\R^d$
\be\label{eq:est_geq}
v(t,x)\geq  v_h(t,x).
\ee
{Under assumptions (H1)-(H3), an upper bound of order $1/6$ for the error related to this approximation was first obtained by Krylov in \cite{Kry99}. Recently, this estimate has been improved to the order $1/4$ in \cite{JakoPicaReisi18}, so that one has 
\be\label{eq:est_krylov}
v(t,x)\leq  v_{h}(t,x)+Ch^{1/4}
\ee
for some constant $C$.
{While the estimates in \cite{Kry99} and \cite{JakoPicaReisi18} are obtained for bounded $\mu,\sigma$ and $\psi$, it follows by similar but more tedious steps that the results also  hold in our framework taking a constant $C$ growing polynomially in the space variable, as already remarked in \cite{JakoPicaReisi18}.} 

The DPP for the value function $v_h$ reads
\begin{equation}\label{eq:DPP_v_delta}
v_h({t_n},x)= \underset{a\in A}\sup \;\E_{t_n,x}\left[v_h({ t_{n+1}},X^{a}_{t_{n+1}})\right].
\end{equation}
In particular,  the restriction of the control set to $\cA_h$ implies that the supremum in \eqref{eq:DPP_v_delta} is taken over the set of control values $A$ (compared with \eqref{eq:DPP_v}). The family of schemes we consider are recursively defined by an approximation of \eqref{eq:DPP_v_delta} and lead to the definition of a numerical solution $V$ {approximating} $v_h$.

\section{Markov chain approximation schemes}\label{sect:disc_scheme}
We  present a class of schemes which are based on a Markov chain approximation of the optimal control problem \eqref{eq:value_v_delta}. This follows the classical philosophy presented in \cite{KD01}, although they take the opposite direction
and use finite difference approximations to construct Markov chains, while here we use time stepping schemes and quadrature formulae
to define SL schemes. Similar probabilistic interpretations of such schemes have been given in  \cite{CF95, Kry99, debrabant2013semi}
for the time-dependent case and in \cite{M89} for the infinite horizon case. What is new here is the construction of schemes with provable higher order error bounds, and the direct use of the dynamic programming principle for the discrete approximation to derive these bounds.

\subsection{Euler-Maruyama scheme}
We start with an approximation of  the process $X^{t,x,\alpha}_\cdot$ by  the Euler-Maruyama scheme. For any  given $\alpha\equiv(a_0,\ldots,a_{N-1})\in \mathcal A_{h}$, we consider the following recursive relation:
\begin{equation}\label{eq:def_M}
X_{t_{i+1}}= X_{t_{i}}+\mu(t_i,X_{t_{i}},a_i)\,h+\sigma(t_i,X_{t_{i}},a_i) \,\Delta B_i
\end{equation}
for $i=0,\ldots, N-1$. The increments  
$\Delta B_i:=( B_{t_{i+1}}- B_{t_i})$
are independent, identically distributed random variables such that
\be\label{eq:N_0_h}
\Delta B_i \sim \sqrt{h} \, \mathcal N(0,I_p)\qquad \qquad \forall i=0,\dots,N-1.
\ee 
We will denote by $\widetilde X^{t_n,x,\alpha}_{\cdot}$ the solution to \eqref{eq:def_M} associated with the control $\alpha\equiv(a_n,\ldots,a_{N-1})\in\cA_{h}$ and such that $\widetilde X^{t_n,x,\alpha}_{t_{n}}=x$.
Under assumptions (H1)-(H2), the rate of strong convergence of the scheme \eqref{eq:def_M} is $1/2$, as given, e.g., in \cite{KloPla}. 
Although the result from there is not directly applicable here as the coefficients are non-Lipschitz in time due to the jumps in
the control process, we can follow the same steps as in the proof of \cite[Theorem 1.1, Chapter I]{Milstein_book},
using the fact that the controls $\alpha\in\cA_h$ are constant over individual timesteps. 
Therefore, one has:
\begin{prop}\label{prop:strong_conv}
Let assumptions (H1)-(H2) be satisfied. Then there exists a constant $\widetilde C\geq 0$ (independent of $h$) such that for any $\alpha\in\cA_h, n=0,\ldots, N, x\in\R^d$, one has
$$
\E_{t_n,x}\left[\,\left|X^{\alpha}_T-\widetilde X^{\alpha}_T\right|\,\right]\leq \widetilde C(1+|x|)h^{1/2}.
$$
\end{prop}
 {For completeness,  a sketch of the proof is reported in the appendix. }\\

As a consequence, denoting
$$
\widetilde v(t_n,x):=\sup_{\alpha\in\cA_{h}} \E_{t_n,x}\left[\psi(\widetilde X^{\alpha}_T)\right],
$$
for any $n=0,\ldots,N-1, x\in\R^d$, thanks to the Lipschitz continuity of $\psi$, one has 
\begin{align}
\left|v_{h} (t_n,x)- \widetilde v (t_n,x)\right| & \leq \underset{\alpha\in\mathcal A_{h}}{\sup}\left|\E_{t_n,x}\left[\psi(X^{\alpha}_T) - \psi(\widetilde X^{\alpha}_T)\right]\right|\leq L \widetilde C(1+|x|)h^{1/2}.\label{eq:EM_est}
\end{align}
Moreover, $\widetilde v$ still satisfies a DPP,
\begin{equation}\label{eq:DPP_v_M}
\widetilde v(t_n,x)= \underset{a\in A}\sup \;\E_{t_n,x}\left[\widetilde v(t_{n+1},\widetilde X^{a}_{t_{n+1}})\right],\qquad \text{ $n=0,\ldots, N-1$}.
\end{equation}

\subsection{Gau{\ss}-Hermite quadrature}
Recalling that $\Delta B_i \sim \sqrt{h} \, \mathcal N(0,I_p)$, we can also write \eqref{eq:DPP_v_M} as 
\begin{equation}\label{eq:DPP_v_M_integral}
\widetilde v(t_n,x)= \underset{a\in A}\sup \;\int_{\R^p} \widetilde v\Big(t_{n+1},x+\mu(t_n,x,a)h+\sqrt{h}\sigma(t_n,x,a) y\Big) \frac{e^{-\frac{|y|^2}{2}}}{(2\pi)^{p/2}} \, dy.
\end{equation}
\normalsize
The discrete-time scheme we are going to define is based on the Gau{\ss}-Hermite approximation of the right-hand term in \eqref{eq:DPP_v_M_integral}. 
Let us start for simplicity with the case $p=1$. \\
 {Let $M\geq 2$ and $H_{_M}$ be the Hermite polynomial of order $M$, i.e.\ 
$$ 
H_{_M}(z) = (-1)^M e^{z^2} \frac{\mathrm d^M}{\mathrm d z^{M} } e^{-z^2} = \sum^{\left \lfloor \frac{M}{2} \right \rfloor}_{k=0}(-1)^k  \frac{M!}{k! (M-2k)!} (2z)^{M-2k}
$$
(see for instance \cite[Section 7.8]{Hildebrand56} for this definition and the following results).
}
We denote by $\{z_i\}_{i=1,\ldots, M}$ the zeros of $H_{_M}$ and by $\{\omega_i\}_{i=1,\ldots ,M}$ the corresponding weights, given by
$$
\omega_i=\frac{2^{M-1} M!\sqrt{\pi}}{M^2 [H_{_{M-1}}(z_i)]^2},\qquad i=1,\ldots,M
$$
Therefore, defining 
$$
\lambda_i:=\frac{\omega_i}{\sqrt{\pi}}\quad\text{and}\quad \xi_i:=\sqrt{2}z_i, \qquad i=1,\ldots, M,
$$
for any smooth real-valued function $f$ (say $f$ at least $C^{2M}$) we can make use of the following approximation  {(see \cite[p.~395]{Hildebrand56})}:
\begin{align}\label{eq:GH_approx}
 \int^{+\infty}_{-\infty} f(y) \frac{e^{-\frac{y^2}{2}}}{\sqrt{2\pi}}dy= \int^{+\infty}_{-\infty} f(\sqrt{2}y) \frac{e^{-y^2}}{\sqrt{\pi}}dy\approx \sum^M_{i=1} \frac{1}{\sqrt{\pi}}\omega_i f(\sqrt{2}z_i)=\sum^M_{i=1} \lambda_i f(\xi_i).
 \end{align}
  {Indeed, 
 the quadrature formula in \eqref{eq:GH_approx}, is exact when the function $f$ is a polynomial of degree lower or equal to $2M-1$}.

 Now observe first that $\lambda_i\geq 0, \forall i=1,\ldots,M$.
 { Moreover, setting $f\equiv 1$ in \eqref{eq:GH_approx} and using that equality holds in this case, one gets $\sum^M_{i=1} \lambda_i=1$.}  
We can therefore define a sequence of i.i.d.\ random variables $\{\zeta_n\}_{n=0,\ldots,N-1}$ such that for any $n=0,\ldots,N-1$
$$
  \P(\zeta_n=\xi_i)=\lambda_i,\quad i=1,\dots,M.
$$
 {Using the fact that the quadrature formula integrates linear and quadratic functions exactly with respect to the  Gau{\ss}ian measure for $M\ge 2$,}
 we have
 $
 \E[\zeta_n]=0$ and $\text{Var}[\zeta_n]=1,\; \forall n=0,\ldots,N-1.
 $
 

\begin{center}
\begin{table}
\caption{Some  $\{(\xi_i,\lambda_i)\}_{i=1,\ldots,M}$ for $M=2,3,4$. We refer to \cite[p.\ 464]{Beyer87} for 
larger $M$.}\label{tab:lambda_xi}
{\begin{tabular}{|c||c|c||c||c|c||c||c|c|}
\hline
 &  $\xi_i$ & $\lambda_i $ & &  $\xi_i$ & $\lambda_i $ & &  $\xi_i$ & $\lambda_i $ \\
\hline  
\hline
$M=2$  & $\pm 1$ &  $1/2$ & $M=3$  & 0 & $2/3$ & $M=4$  & $\pm \sqrt{3-\sqrt{6}}$  &  $(3+\sqrt{6})/12$ \\
&&&& $\pm\sqrt{3}$ & $1/6$ && $\pm \sqrt{3+\sqrt{6}}$  &   $(3-\sqrt{6})/12$ \\
\hline
\end{tabular}
}
\end{table}
\end{center}

For any control $\alpha\equiv(a_n,\ldots,a_{N-1})\in\cA_{h}$, in the sequel  we will denote by $\widehat X^{t_n,x,\alpha}_\cdot$ 
 the Markov chain approximation of the process $\widetilde X^{t_n,x,\alpha}_\cdot$  recursively defined   by
 \be\label{eq:def_Xh}
 \widehat X_{t_{i+1}} =  \widehat X_{t_{i}}+\mu(t_i,\widehat X_{t_{i}},a_i)\,h+\sqrt{h}\sigma(t_i,\widehat X_{t_{i}},a_i)\,\zeta_i,  \quad\text{for $i=n,\ldots,N-1$ }
 \ee
 with   $\widehat X_{t_n}  =  x$.
Therefore, starting from \eqref{eq:DPP_v_M_integral} and applying the Gau{\ss}-Hermite quadrature formula \eqref{eq:GH_approx}, our scheme will be defined by 
 \small
 \begin{equation}
 \label{eq:SL_V}
\hspace{-0.5cm} \left\{ 
 \begin{array}{rl}
 \widehat v(t_n,x)= &\hspace{-0.3cm}\underset{a\in A}\sup\; \sum^M_{i=1}\lambda_{i} \widehat v\left(t_{n+1},x+\mu(t_n,x,a)h+\sqrt{h}\sigma(t_n,x,a)\xi_i\right) \\
  =&\hspace{-0.3cm}\underset{a\in A}\sup\;\E_{t_n,x}\left[ \widehat v(t_{n+1},\widehat X^{a}_{t_{n+1}})\right], \hfill n=N-1,\ldots,0, \\
 \widehat v(t_{_N},x)=& \hspace{-0.3cm}\psi(x).
 \end{array} 
 \right.
 \end{equation}
 \normalsize
\begin{remark}
For $M=2$, \eqref{eq:SL_V} is the SL scheme introduced by Camilli and Falcone in \cite{CF95}, for now without considering interpolation on any spatial grid.
\end{remark}
 Iterating, we obtain the following representation formula for $\widehat v$:
$$
\widehat v(t_n,x)=\underset{\alpha\in\mathcal A_{h}}{\sup}\E_{t_n,x}\left[\psi(\widehat X^{\alpha}_T)\right].
$$
\normalsize
 \subsubsection*{The rate of weak convergence}\label{sect:weak}
\noindent
In this section, we prove the rate of weak convergence of the random walk $\widehat X^{\alpha}_\cdot$ defined by \eqref{eq:def_Xh} to the process $\widetilde X^{\alpha}_\cdot$ given by the Euler-Maruyama scheme \eqref{eq:def_M}.
\begin{prop}\label{prop:weak_conv}
Let assumptions (H1)-(H2) be satisfied and let $M\geq 2$. Then there exists a constant $\widehat C\equiv \widehat C(M)\geq 0 $ such that for any function  $f\in C^{2M}(\R^d;\R)$ one has
$$
\Big|\;\E_{t_n,x}\Big[f(\widetilde X^{a}_{t_{n+1}})\Big]-\E_{t_n,x}\Big[f(\widehat X^{a}_{t_{n+1}})\Big]\;\Big|\leq \widehat C {\|D^{(2M)} f\|_{\infty}}
{( 1 + |x|^{2M})} h^{M},
$$
for any $x\in\R^d$, $a\in A$, $h\geq 0$ and $n=0,\ldots,N-1$, {and where we denoted for $k\in\mathbb N$ 
$$\|D^{(k)} f\|_\infty :=\sup_{\substack{z\in \R^d\\ \beta\in\mathbb N^d, |\beta|=k}} \left|\frac{\partial^{|\beta|} f(z)}{\partial x_1^{\beta_1}\ldots \partial x_d^{\beta_d}}\right|.$$
}
\end{prop}
\begin{proof}
{We adapt a standard argument from numerical quadrature.}
{Let us take for simplicity $d=1$ (the case $d>1$ works in the same way)} and denote 
$
z=x+h\mu(t_n,x,a).
$
By Taylor expansion, we can write
\small
\begin{align*}
 \E_{t_n,x}\Big[f(\widetilde X^{a}_{t_{n+1}})\Big] 
& =\int^{+\infty}_{-\infty} f(z+\sqrt{2h}\sigma(t_n,x,a)y) \frac{e^{-y^2}}{\sqrt{\pi}}dy \\
& =\int^{+\infty}_{-\infty} \bigg\{\sum^{2M-1}_{k=0}\frac{f^{(k)}(z)}{k!}(\sqrt{2h}\sigma(t_n,x,a)y)^k+\frac{f^{(2M)}(\hat z)}{(2M)!}(\sqrt{2h}\sigma(t_n,x,a)y)^{2M}\bigg\} \frac{e^{-y^2}}{\sqrt{\pi}}dy,
\end{align*}
\normalsize
for some $\hat z$. In the same way we get
\small 
\begin{align*}
\E_{t_n,x}\Big[f(\widehat X^{a}_{t_{n+1}})\Big]
=\sum^M_{i=1} \frac{\omega_i}{\sqrt{\pi}} \bigg\{\sum^{2M-1}_{k=0}\frac{f^{(k)}(z)}{k!}(\sqrt{2h}\sigma(t_n,x,a) \xi_i)^{k}+\frac{f^{(2M)}(\tilde z)}{(2M)!}(\sqrt{2h}\sigma(t_n,x,a) \xi_i)^{2M}\bigg\},
\end{align*}
\normalsize
for some $\tilde z$. At this point we recall that, by construction, the Gau\ss-Hermite quadrature formula is exact for any polynomial of degree $\leq 2M-1$, so for any $k\in\{0,\dots,2M-1\}$ we have
$$
\frac{1}{\sqrt{\pi}}\frac{f^{(k)}(z)}{k!}(\sqrt{2h}\sigma(t_n,x,a))^{k}\bigg\{\int^{+\infty}_{-\infty}y^{k}{e^{-y^2}}dy-\sum^M_{i=1}\omega_i z_i^{k}\bigg\}=0.
 $$
This implies that 
\small
\begin{align*}
& \Big|\;\E_{t_n,x}\Big[f(\widetilde X^{a}_{t_{n+1}})\Big]-\E_{t_n,x}\Big[f(\widehat X^{a}_{t_{n+1}})\Big]\;\Big|\\
& \leq \Big |\int^{+\infty}_{-\infty}\frac{f^{(2M)}(\hat z)}{(2M)!}(\sqrt{2h}\sigma(t_n,x,a)y)^{2M} \frac{e^{-y^2}}{\sqrt{\pi}}dy-\sum^M_{i=1}\frac{\omega_i}{\sqrt{\pi}}\frac{f^{(2M)}(\tilde z)}{(2M)!}(\sqrt{2h}\sigma(t_n,x,a) z_i)^{2M}\Big|\\
& \leq \widehat C \|f^{(2M)}\|_{\infty} h^{M} (1+|x|^{2M}),
\end{align*}
\normalsize
  {where we have used the fact that $|\sigma(t,x,a)|\leq C_1 (1+|x|)$ for some $C_1\geq 0$. Calculating explicitly in the second line
 $\int^{+\infty}_{-\infty} 2^M y^{2M}\frac{e^{-y^2}}{\sqrt{\pi}}dy = (2M-1)!!$ and by the fact that $C_1$
 depends only on $C_0$ in (H2), the constant $\widehat C$ can be given as
$$
\widehat C = \frac{2^{2M-1} }{(2M)!} C_1^{2M}\Big( 2(2M-1)!! + \Big| (2M-1)!! - \sum^M_{i=1} \lambda_i \xi^{2M}_i\Big|\Big)
$$
}
and only depends on $M$ and the constants in assumption (H2).
 \end{proof}

\subsubsection*{Multi-dimensional Brownian motion}
In the case of $p>1$,  {it is classical (see, e.g.\ \cite[Section 5.6]{davis2007methods})} to define an approximation by a tensor product of the formula \eqref{eq:GH_approx},
that is 
\begin{align}\label{eq:GH_approx_multidim}
 \int_{\R^p} f(\sqrt{2}y) \frac{e^{-{|y|^2}}}{{\pi}^{p/2}}dy\approx \sum^M_{i_1,\ldots,i_p=1} \lambda_{i_1}\cdots\lambda_{i_p} f(\xi_{i_1},\ldots,\xi_{i_p}).
 \end{align}
Then, {denoting  for any $i\equiv(i_1,\ldots,i_p)\in \{1,\ldots, M\}^p$ the vector $\xi_i\equiv (\xi_{i_1},\ldots, \xi_{i_p})^\top\in \R^p$ and the scalar $\lambda_i=\lambda_{i_1}\cdots\lambda_{i_p}\in \R$}, one can define an approximation to $v$ by
 \small
 \begin{equation}
 \label{eq:SL_V_dim}
\hspace{-0.5cm} \left\{
 \begin{array}{rl}
 \widehat v(t_n,x)= &\hspace{-0.3cm}\underset{a\in A}\sup\; \sum_{i\in \{1\ldots M\}^p}\lambda_{i} \widehat v \Big(t_{n+1},x +\mu(t_n,x,a)h+\sqrt{h}\sigma(t_n,x,a)\; \xi_i\Big) \\
  =&\hspace{-0.3cm}\underset{a\in A}\sup\;\E_{t_n,x}\Big[ \widehat v(t_{n+1},\widehat X^{a}_{t_{n+1}})\Big], \hfill n=N-1,\ldots,0, \\
 \widehat v (t_{_N},x)=& \hspace{-0.3cm}\psi(x).
 \end{array} 
 \right.
 \end{equation}
 \normalsize
 
  {
This differs from the traditional Markov chain approximation approach taken in \cite[Section 5.3]{KD01}, where monotone finite difference schemes (i.e., those leading to positive weights)
are used to define the matrix of transition probabilities.}
 
{It is easy to observe that the construction in (\ref{eq:GH_approx_multidim}) leads to an exponential growth of the computational complexity in the dimension $p$,
as it requires at each time step and for each node the evaluation of the solution at $M^p$ points.
Retracing the proof of Proposition \ref{prop:weak_conv}, one can deduce that in order to guarantee a weak error estimates of order $h^M$, it is sufficient to find weights 
$\hat\lambda_i$ 
and nodes 
$\hat \xi _i$, $i=1,\ldots, \hat M$,
for some $\hat M\in\N$ possibly lower than $M^p$, which integrate exactly all polynomials of degree lower or equal than $2M-1$.
Moreover, the probabilistic interpretation of our scheme also requires that $\hat{\lambda}_i\geq 0$, $i=1,\ldots, \hat M$.
Such pairs $\{(\hat{\lambda}_i,\hat{\xi}_i)\}_{i=1,\ldots, \hat M}$ then have to satisfy
$$
\mathbf{A}  \lambda = \mathbf{b},\qquad \lambda \geq 0,
$$
with  $\mathbf{A}\in \R^{\ell\times M^p}$ and $\mathbf{b}\in \R^\ell$ defined by 
$$
\mathbf{A}=\left(\begin{array}{ccc}
\gamma_1(\xi_1) & \dots & \gamma_1(\xi_{M^p})\\
\vdots & \vdots & \vdots\\
\gamma_\ell(\xi_1) & \dots & \gamma_\ell(\xi_{M^p})
\end{array}
\right)\quad\text{and}\quad {\mathbf b}_i=\int_{\R^p} \gamma_i(\sqrt{2}y)\frac{e^{-|y|^2}}{\pi^{p/2}}\ dy, \quad i=1,\ldots, \ell,
$$
where $\{\gamma_1,\ldots,\gamma_\ell\}$ is a basis for the space of polynomials of degree $2M-1$ in $\R^p$ and $\ell=\binom{2M-1+p}{p}$.

The existence of a solution of the form  $\widehat\lambda=(\hat\lambda_1,\ldots, \hat{\lambda}_{\hat M}, 0,\ldots, 0 )$ for some  $\hat M\leq \ell$
follows from Tchakaloff's Theorem (see \cite{tchakaloff1957formules}, and also \cite{bayer2006proof} for a recent simpler proof).

A constructive method for independent Gau{\ss}ian random variables as in the present case is proposed in \cite{devuyst2007gaussian}, while an efficient procedure for the general, dependent case applied to the uniform measure is given in \cite{tchernychova2015caratheodory}.
This gives a substantial reduction for large $p$ and moderate $M$ in particular.
Table 4.1 in \cite{tchernychova2015caratheodory} gives numerical values for $\ell$ versus $M^p$ for $M=3$ and different $p$, such as:
$p=2$: $\ell = M^p = 9$; $p=3$: $\ell = 23, M^p = 27$; $p=5$: $\ell = 96,  M^p = 243$; $p=10$: $\ell = 891, M^p = 59049$.

 {
We end by noting that sparse grid quadrature (see, e.g.\ \cite{gerstner1998numerical}), which has been shown to overcome the curse of dimensionality for integrals of sufficiently regular functions,
is not suitable here because of the negativity of weights which is essential to their construction.
On the other hand, applying cubature on Wiener space (see \cite{lyons2004cubature}) may be a possible extension.
}

{In what follows, we will use the notation $\{({\hat \lambda}_i,{\hat \xi}_i)\}_{i=1,\ldots, \hat M}$ to generalise \eqref{eq:SL_V} to any $p\geq 1$.}

\subsubsection*{Lipschitz regularity of approximation}
We conclude this section with a regularity result for $\widehat v$. This is an important property of our scheme strongly exploited in Proposition \ref{lem:fully_est} and Section \ref{sec:lower}. 
\begin{prop}\label{prop:lip_V}
Let  (H1)-(H3) be satisfied. There exists $C\geq 0$  such that 
$$
|\widehat v(t_n,x) - \widehat v(t_n,y)|\leq LC |x-y|
$$
for any $x,y\in\R^d$ and $n=0,\ldots,N$ (where $L$ is the Lipschitz constant of $\psi$ and $C$ only depends on $T$ and the constant $C_0$ in Assumption (H2)).  
\end{prop}
\begin{proof}
The result can be proved by backward induction in $n$. For $n=N$, $\widehat v(t_{_N},\cdot)$ is Lipschitz with constant $L_N:=L$ given by (H3). Let $\widehat v(t_{i},\cdot)$ be Lipschitz continuous with constant $L_{i}$ (only depending on $T$ and $C_0$ in Assumption (H2)) for any $i=n+1,\ldots,N$. By classical estimates and thanks to the definition of $(\hat\lambda_i,\hat \xi_i)$ such that $\sum^{\hat M}_{i=1} \hat{\lambda}_i\hat{\xi}_i = 0$ and $\sum^{\hat M}_{i=1} \hat{\lambda}_i|\hat{\xi}_i|^2 = p$, {one can show by a straightforward calculation that  
\begin{align*} 
&\E\left[\,\left|\widehat X^{t_n,x,a}_{t_{n+1}}-\widehat X^{t_n,y,a}_{t_{n+1}}\right|^2\,\right]
\leq \left(1+Ch \right) |x-y|^2.
\end{align*}
}
Hence, by the definition of $\widehat v$  one has
\begin{align*}
\left|\widehat v(t_n,x)-\widehat v(t_n,y)\right| \leq L_{n+1} \E\left[\,\left|\widehat X^{t_n,x,a}_{t_{n+1}}-\widehat X^{t_n,y,a}_{t_{n+1}}\right|\,\right] \leq L_{n+1} (1+Ch)^{1/2} |x-y|,
\end{align*}
where $C$ only depends on $C_0$ in (H2), which gives $
L_n\leq L_{n+1}(1+Ch)^{1/2}.
$
Iterating, one obtains
$
L_n \leq L (1+Ch)^{N/2} \leq L e^{ChN/2} \leq L e^{C T},
$
which concludes the proof.
\end{proof}
}

\subsection{The fully discrete scheme}
In order to be able to compute the numerical solution practically in reasonable complexity, {we need to introduce some sort of recombination,}
otherwise the total number of nodes of all trajectories grows exponentially in $N$. 

Let ${\Delta x\equiv(\Delta x_1,\dots,\Delta x_d)\in (\R^{>0})^d}$ and consider the space grid  $\mathcal G_{\Delta x}:=\{x_m = m\Delta x : m\in{\Z^d}\}$.
Let $\mathcal I[\cdot]$ denote the standard multilinear interpolation operator {with respect to the space variable} which 
satisfies for every Lipschitz function $\phi$ (with Lipschitz constant $L_\phi$):
\begin{subequations}
\begin{eqnarray}
&& \mathcal I[\phi](x_m)=\phi(x_m),\quad \forall m\in\Z^d, \label{eq:interp1}\\
&& |\mathcal I[\phi](x)-\phi(x)| \leq L_\phi |\dx|, \label{eq:interp2}\\
&& \text{for any functions } \phi_1,\phi_2:\R^d\conv\R,\quad \phi_1\leq \phi_2 \Rightarrow \mathcal I[\phi_1]\leq \mathcal I[\phi_2].\label{eq:interp3}
\end{eqnarray}
\end{subequations}

{
We define an approximation on this fixed grid, denoted by $\VF$, by:}
 \begin{equation}
 \label{eq:SL_V_fully}
\hspace{-0.5cm} \left\{
 \begin{array}{rl}
 \VF(t_n,x_m)= &\hspace{-0.3cm}\underset{a\in A}\sup\;\sum^{{ \hat M}}_{i=1} {\hat \lambda_i}\; \mathcal I[\VF]\Big(t_{n+1},x_m +\mu(t_n,x_m,a)h+\sqrt{h}\sigma(t_n,x_m,a)\; {\hat \xi_i}\Big),\\
 \VF(t_{_N},x_m)=& \hspace{-0.3cm}\psi(x_m),
 \end{array} 
 \right.
 \end{equation}
 for $n=N-1,\ldots,0$ and $m\in \Z^d$.
{We will refer to this as the fully discrete scheme.}
 
 From the properties of multilinear interpolation, for all $x\in \mathbb{R}^d$ there exist $q_k(x)\ge 0$, $k\in \Z^d$ with $\sum_k q_k(x)=1$ and
 $|\{k: q_k > 0\}| \le 2^p$ such that
 $
 \mathcal I[\phi](x) = \sum_{k\in\Z^d} q_k(x) \phi(x_k).
 $
Then with (\ref{eq:SL_V_fully}),
 \[
\VF(t_n,x_m)= \underset{a\in A}\sup\;\sum_{k\in\Z^d} {\hat \lambda_{m,k}(t_{n},a)}
 \; \VF(t_{n+1},x_k)
 \]
 with  $\hat \lambda_{m,k}(t_{n},a) := \sum_{i=1}^{{ \hat M}} {\hat \lambda_i}\; q_k\big(x_m +\mu(t_n,x_m,a)h+\sqrt{h}\sigma(t_n,x_m,a)\; {\hat \xi_i}\big) \ge 0$ and
$\sum_k \hat \lambda_{m,k}(t_{n},a) =1$.
  
 Therefore, $\hat \lambda_{m,k}(t_{n},a)$ are interpretable as transition probabilities of a controlled Markov chain with state space $\mathcal G_{\Delta x}$.
 The number of transitions from node $m$ is $|\{k: \hat \lambda_{m,k}(t_{n},a)>0\}| \le 2^{d} \ell$. 
 
 \begin{prop}\label{lem:fully_est}
 Let assumptions (H1)-(H3) be satisfied. Then, there exists $C\geq 0$ such that 
 $$
 \sup_{\substack{n=0,\ldots, N,\\ m\in\Z^d}} |\widehat v(t_n,x_m)-\VF(t_n,x_m)| \leq C \frac{|\Delta x|}{h}.
 $$
 \end{prop}
 \begin{proof}
 The result follows by properties \eqref{eq:interp2}-\eqref{eq:interp3} and by the Lipschitz continuity of $\widehat v$ proved in Proposition \ref{prop:lip_V} (see also \cite[Lemma 7.1]{DJ12}).
 \end{proof}
Observe that, in absence of further regularity assumptions, this introduces the following ``inverse CFL condition'' for the convergence of the fully discrete scheme:
$|\Delta x|/h\to 0$ as  $|\Delta x|, h\to 0.$

\section{Error estimates}\label{sec:lower}

In order to obtain error estimates for the scheme described in Section \ref{sect:disc_scheme}, we will adapt the  technique of ``shaking coefficients'' and regularization introduced by Krylov in \cite{K97,K00} and studied later by many authors  (see for instance \cite{BJ02,BJ05,BJ07}) for obtaining the rate of convergence of monotone numerical scheme for second order HJB equations.
We do so without passing through the PDE consistency error and work instead with the direct estimates we presented in the previous section. 

We refer to Section \ref{sec:smooth_case} for a discussion of the regular case.

\subsection{Regularization}\label{sec:regulariz}
Let $\veps>0$ and  let $\cE_h$
be the set of $\R^d$-valued progressively measurable processes 
$e$ bounded by $\vare$ which are constant in each time interval $[t_i,t_{i+1}$],
that is, 
$$
  \cE_h:=\Big\{ e, \text{progr. meas.:} \forall \omega\in \Omega {\,\exists e_i\in \R^d,\,  |e_i|\leq \vare, i=0,\ldots, N-1 \text{ s.t. } e_s(\omega)= \sum^{N-1}_{i=0} e_i \mathbbm 1_{s\in [t_i,t_{i+1})}}\Big\}.
$$ 
For any pair $(\alpha,e)\in\cA_h\times \cE_h$, let us consider the process $\widetilde X^{t_n,x,\alpha,e}_\cdot$ defined by the following $\vare$-perturbation of the dynamics \eqref{eq:def_M}:
\be\label{eq:def_dyn_eps}
 \widetilde X_{t_{i+1}}= \widetilde X_{t_{i}}+\mu(t_i,\widetilde X_{t_{i}}+e_i,a_i)\,h+\sigma(t_i,\widetilde X_{t_{i}}+e_i,a_i)\,\Delta B_i,
 \ee
for $i=n,\ldots, N-1$ with $\widetilde X_{t_n} = x$. We define  the following ``perturbed'' value function: 
\be\label{eq:def_veps}
 v^\vare(t_n,x):=\underset{\alpha\in\mathcal A_{h}, e\in\cE_h}\sup\E_{t_n,x}\Big[\psi( \widetilde X^{\alpha,e}_T)\Big]\qquad\text{ $n=0,\ldots,N$, $x\in\R^d$}.
\ee 
\begin{prop}\label{prop:v-vE}
Let assumptions (H1)-(H3) be satisfied. Then there exists  a constant $C\geq 0$ such that for any $n=0,\dots,N$ and $ x,y\in\R^d$
$$
|v^\vare (t_n,x)-v^{\vare}(t_n,y)|\leq L C |x-y|
\quad \text{and}\quad
|\widetilde v(t_n,x)-v^{\vare}(t_n,x)|\leq L C\vare.
$$
\end{prop}
\begin{proof}
The Lipschitz continuity of $v^\vare$ follows by  the standard estimate
$$
\E\left[\underset{i=n,\ldots,N}\sup\left|\widetilde X^{t_n,x,\alpha,e}_{t_i}-\widetilde X^{t_n,y,\alpha,e}_{t_i}\right|\right]\leq C |x-y|, \qquad\text{$n=0,\ldots,N$, $x,y\in\R^d$}.
$$
Let us fix a control $\alpha\in\mathcal A_{h}$ and $e\in\cE_h$.
For any $i=n,\ldots,N-1$, 
By the definition of processes \eqref{eq:def_M} and \eqref{eq:def_dyn_eps} one has for any $i=n,\ldots,N-1$
\begin{align*}
\widetilde X^{t_n,x,\alpha}_{t_{i+1}}-\widetilde X^{t_n,x,\alpha,e}_{t_{i+1}} = &\;   \Big(\widetilde X^{t_n,x,\alpha}_{t_{i}}-\widetilde X^{t_n,x,\alpha,e}_{t_{i}}\Big)+h\Big(\mu(t_i,\widetilde X^{t_n,x,\alpha}_{t_{i}},a_i)-\mu(t_i,\widetilde X^{t_n,x,\alpha,e}_{t_{i}}+e_i,a_i)\Big)\\ &+\Big(\sigma(t_i,\widetilde X^{t_n,x,\alpha}_{t_{i}},a_i)-\sigma(t_i,\widetilde X^{t_n,x,\alpha,e}_{t_{i}}+e_i,a_i)\Big)\Delta B_i.
\end{align*}
{ Using the fact that $\E[\Delta B_i] = 0$ and $\E[|\Delta B_i|^2] = p h$,  and  the Lipschitz continuity of $b$ and $\sigma$ (assumption (H2)),} a  straightforward calculation shows that 
\begin{align*}
 \E_{t_n,x}\Big[\big|\widetilde X^{\alpha}_{t_{i+1}}-\widetilde X^{\alpha,e}_{t_{i+1}}\big|^2\Big]
 \leq \Big(1+ Ch\Big)\E_{t_n,x}\Big[\big|\widetilde X^{\alpha}_{t_{i}}-\widetilde X^{\alpha,e}_{t_{i}}\big|^2\Big]+\vare^2Ch,
\end{align*}
with $C$ a positive constant independent of $\alpha, e$ and $h$. 
By iteration we finally get 
\begin{align*}
\E_{t_n,x}\Big[\big|\widetilde X^\alpha_{t_{i}}-\widetilde X^{\alpha,e}_{t_{i}}\big|^2\Big]\leq  {\vare^2  \sum^{i-1}_{k=n}C h (1+Ch)^k   } \leq   \vare^2 CT e^{CT},
\end{align*}
for any $i=n,\ldots,N$ and we can conclude that there exists $C\geq 0$ such that
\[
|\widetilde v(t_n,x)-v^{\vare}(t_n,x)|\leq \underset{\alpha\in \mathcal A_{h},e\in\cE_h}{\sup} \E_{t_n,x}\left[\,\left|\psi(\widetilde X^{\alpha}_{T})-\psi(\widetilde X^{\alpha,e}_{T})\right|\,\right]\leq L C\vare.
\]
\end{proof}
{We point out that for the perturbed value function $v^\vare$ the following DPP holds}:
\begin{equation}\label{DPP_vE}
v^\vare(t_n,x)= \underset{a\in A,|e|\leq \vare}\sup \E_{t_n,x}\Big[v^\vare(t_{n+1}, \widetilde X^{a,e}_{t_{n+1}})\Big],\qquad n=0,\ldots,N-1.
\end{equation}
The step that follows consists in a regularization of the function $v^\vare$. 
We  consider a smooth function  
$\delta:\R^d\rightarrow [0,+\infty)$ supported in the unit ball $B_1(0)$ with $\int_{\R^d}\delta(x)\,dx=~1$, and 
 we define $\{\delta_\veps\}_{\vare>0}$ as the 
 sequence of mollifiers
$
\delta_\vare(x):=\vare^{-d}\delta\left(x/\vare\right).
$
Then define, for any $n=0,\ldots,N$, 
\be\label{def:ve}
v_{\vare}(t_n,x):=\int_{\R^d} v^{\vare}(t_n,x-\xi)\delta_\vare(\xi)\mathrm{d}\xi.
\ee

\begin{prop}\label{prop:DPP_ve}
Let assumptions (H1)-(H3) be satisfied. Then,
\begin{enumerate}
\item[(i)] there exists  $C\geq 0$ such that 
$$
\big|v_\vare(t_n,x)-v^\vare(t_n,x)\big|\leq L C\vare\qquad n=0,\ldots,N,\, x\in\R^d;
$$
\item[(ii)] the function $v_\vare(t_n,\cdot)$  is $C^\infty$ for $n=0,\ldots,N$ and for any $k\geq 1$ there is $C\geq 0$ such that
\be\label{eq:est_deriv}
{
\underset{n=0,\ldots,N}\sup\big\|D^{(k)} v_\vare(t_n,\cdot)\big\|_\infty \leq L C \varepsilon^{1-k};}
\ee
\item[(iii)] $v_\vare$ satisfies the following super-dynamic programming principle 
\begin{equation}\label{DPP_ve}
v_\vare(t_n,x)\geq  \underset{a\in A}\sup\; \E_{t_n,x}\Big[v_\vare(t_{n+1}, \widetilde X^{a}_{t_{n+1}})\Big],\qquad n=0,\dots,N-1,\, x\in\R^d.
\end{equation}
\end{enumerate}
\end{prop}
\begin{proof}
Properties $(i)$-$(ii)$ follow by the properties of mollifiers and the Lipschitz continuity of $v^\vare$ (Proposition \ref{prop:v-vE}).
It remains to prove $(iii)$. By the definition of $v_\vare$, equality \eqref{DPP_vE} and using the fact that for any $a\in A$, $\xi\in B_\varepsilon(0)$, $n=0,\ldots,N-1$, one has
$
\widetilde X^{t_n,x-\xi,a,\xi}_{t_{n+1}}= \widetilde  X^{t_n,x,a}_{t_{n+1}}-\xi$, 
one obtains
\begin{align*}
v_\vare(t_n,x) 
  \geq \int_{\R^d}\;\underset{a\in A}\sup\;\E_{t_n,x-\xi}\Big[v^\vare(t_{n+1}, \widetilde X^{a,\xi}_{t_{n+1}})\Big]\delta_\vare(\xi)d\xi
  \geq  \underset{a\in A}\sup\, \E_{t_n,x}\Big[\int_{\R^d}v^\vare(t_{n+1}, \widetilde X^{a}_{t_{n+1}}-\xi)\delta_\vare(\xi)d\xi\Big],
\end{align*}
which concludes the proof.
\end{proof}

\subsection{{Improved lower bound}}\label{subsect:lower}
Applying \eqref{eq:EM_est}, Proposition \ref{prop:v-vE} and Proposition \ref{prop:DPP_ve}$(i)$, we obtain
\begin{align}\label{eq:lower_first_passage}
& v(t_n,x)\geq
v_\vare(t_n,x)-L \widetilde C(1+|x|)h^{1/2}-L C\vare,
\end{align}
\normalsize
for some new $C\geq 0$.
Moreover, by Proposition  \ref{prop:DPP_ve} ($(ii)$ and $(iii)$) and Proposition \ref{prop:weak_conv} we also have
\begin{align*}
\widehat v(t_n,x)-v_\vare(t_n,x)
 \leq \underset{a\in A}{\sup}\,\E_{t_n,x}\Big[\widehat  v(t_{n+1},\widehat X^{a}_{t_{n+1}})-v_\vare(t_{n+1},\widehat X^{a}_{t_{n+1}})\Big]+L C(1\!+\!|x|^{2M})\vare^{1-2M} h^{M}. 
\end{align*}
We can then iterate this inequality to  get 
\begin{align}\label{eq:loweb}
 \widehat v (t_n,x)-v_\vare(t_n,x)&\leq \|\widehat v(t_{N},\cdot)-v_\vare(t_{N},\cdot)\|_\infty+L C(1+|x|^{2M})\vare^{1-2M} h^{M-1}, 
\end{align}
where we have used that for some $C\geq 0$
$$
 \sup_{\alpha\in \mathcal A_h}\E_{t_n,x}\Big[\underset{i=n,\ldots,N}\sup \big|\widehat X^{\alpha}_{t_i}\big|^{2M}\Big]\leq C \big(1+|x|^{2M}\big).
$$
}
Hence, combining \eqref{eq:loweb} and   \eqref{eq:lower_first_passage}, we can conclude that for any $n=0,\ldots,N, x\in\R^d$
\begin{align*}
v(t_n,x)\geq \widehat v(t_n,x)-L C (1+|x|^{2M})\Big(\vare^{1-2M} h^{M-1}+h^{1/2}+\vare\Big).
\end{align*}
{Balancing the terms with $\vare$ and $h$, i.e. taking $\vare=h^{(M-1)/2M}$, by $1/2> (M-1)/2M$ one has
\begin{align}\label{eq:lowerSemi}
v(t_n,x)\geq \widehat v(t_n,x)-L C (1+|x|^{2M})h^{(M-1)/2M}.
\end{align} }
To conclude,  the interpolation error has to be added giving an overall error of 
$$
O\Big(h^{(M-1)/2M} + \frac{|\Delta x|}{h}\Big).
$$
Optimising the choice of $\Delta x$ with respect to $h$ we get
$
|\Delta x| \sim h^{(3M-1)/2M}.
$
This effectively leads to order $(M-1)/(2M)$ in time and $(M-1)/(3M-1)$ in space, which can be made arbitrarily close to $1/2$ and $1/3$,
respectively, by choosing $M$ large enough.

\begin{remark}[Comparison with existing results]\label{rem:DJ}
By a Taylor expansion it is possible to compute the consistency error of the scheme with respect to the associated HJB equation. Considering, for simplicity, the uncontrolled {one-dimensional ($p=d=1$)}  case with $\mu\equiv 0$,
using the fact that  $\sum_{i=1}^{M} \lambda_i =1$, $\sum_{i=1}^{M} \lambda_i \xi_i^2 =1$
 and  $\sum_{i=1}^{M} \lambda_i \xi_i^{2k+1} =0\ (\forall k\in \N)$, one gets 
\begin{align*}
& \frac{1}{h} \Big(v(t_{n+1},x) - \sum_{i=1}^M \lambda_iv(t_{n},x + \sqrt{h}\sigma(t_n,x)\xi_i) \Big)\\
& =   v_t(t_n,x) - \frac{1}{2}(\sigma(t_n,x))^2 v_{xx}(t_n,x) +  \frac{h}{2} v_{tt}(t_n,x)  -\frac{h}{4!}(\sigma(t_n,x))^4 v_{4x}(t_n,x)\sum_{i=1}^M \lambda_i \xi_i^4  
 +O(h^{2}),
\end{align*}
which shows that the scheme has order $1$ consistency, for all $M$.
Applying the results in \cite{DJ12}, this would lead to error estimates of order $h^{1/4}+\Delta x/h$, 
i.e., with the optimal choice of $\Delta x$
order $1/4$ in $h$ and $1/5$ in $\Delta x$. 
A similar limitation applies to the analysis in \cite{Kry99}.

The improvement we get for the lower bound is due to the fact that, splitting the two contributions of the error coming from Euler-Maruyama time stepping and the Gau\ss-Hermite quadrature formula, we can reduce the second one by increasing $M$, whereas for the first one the lower regularity requirement allows us to get order $1/2$.
\end{remark}

\subsection{{Upper bound}}\label{subsect:upper}
{
The first important observation is that the estimates based on the convexity of the supremum operator (Proposition \ref{prop:DPP_ve}($iii$)) work only in the direction of the lower bound. 
Due to the regularity of the numerical solution (see Proposition \ref{prop:lip_V}), we can apply the approach from \cite{BJ02, K97, K00}
to reverse the role of numerical and exact solution and exploit the same arguments by regularization of $\widehat v$.
However, to estimate the error introduced by the piecewise approximation of the controls we rely on  \eqref{eq:est_krylov}. This  restricts the convergence rate to order 1/4 in $h$ and $1/5$ in $\Delta x$ and hence it will not lead to an improvement with respect to the rates in \cite{DJ12} even for large $M$. }

\section{The regular case and improvements}\label{sec:smooth_case}

\subsection{The regular case}
If the value function $v$ can be shown to be sufficiently smooth, the regularization step is not necessary and {it is also possible
to consider the rate of  weak convergence of the Euler-Maruyama scheme, which is one, and under differentiability assumptions on $\psi$ this gives
$$
\underset{\alpha\in \mathcal A_h}\sup \Big|\E_{t_n,x}\big[\psi(X^\alpha_T) - \psi(\widetilde X^{\alpha}_T)\big]\Big| \leq  C h.
$$
}
Thus, we obtain the following lower estimate  
\be\label{eq:semi_reg}
v(t_n,x)\geq \widehat v(t_n,x)+  C (1+|x|^{2M}) h^{M-1}+  C h,
\ee
which is of order 1 as we would expect in the regular case.
{For  sufficiently smooth functions, the interpolation error reduces to $|\Delta x^2|/h$}. This, together with \eqref{eq:semi_reg}, gives estimates for the lower bound of order $O(h + \frac{|\Delta x|^2}{h})$. %
It is also shown in \cite{JakoPicaReisi18} that $0\leq v-v_h\leq C h$ holds if $v_h$ is sufficiently smooth. For $|\Delta x| \sim h$ this leads to error estimates of  order 1. 
In many cases, this corresponds to the practically observed situation so  that choosing $|\Delta x| \sim h$ is sufficient to observe convergence, with order 1, of the fully discrete scheme.

\subsection{Higher order time stepping}
In the smooth case it can also be beneficial to consider higher order approximation schemes for the stochastic differential equation (in the non-smooth case, the necessity of heavier regularization neutralizes the improvements from the higher order schemes). For instance, in the case of coefficients independent of time, one could adapt  the weak-second order Taylor scheme (see \cite{KloPla})
\begin{align*}
X^{{t_n},x}_{t_{n+1}} = & x +\mu(x)h + \Big( -\frac{1}{2}\sigma\sigma_x(x) +\frac{1}{2}\mu\mu_x(x)+\frac{1}{4}\mu_{xx}\sigma^2(x)\Big) h^2 + \sigma(x)\Delta B_i\\
&  +\frac{1}{2}\sigma\sigma_x(x) \Delta B_i^2 +\Big(\frac{1}{2}\mu_x\sigma(x) +\frac{1}{2}\mu\sigma_x(x) +\frac{1}{4}\sigma_{xx}\sigma^2(x)\Big)h\Delta B_i
\end{align*}
to the controlled equation \eqref{SDE}
and obtain an error contribution of order $h^2$ from the time stepping scheme for the semi-discrete approximation.
{Retracing the steps of the proof of Proposition \ref{prop:weak_conv}, $M\ge 2$ is still sufficient to guarantee order 2 for the Gau{\ss}-Hermite approximation, as the higher order terms resulting from
$B_i^2$ are integrated sufficiently accurately.}
The overall lower bound of the error for the fully discrete scheme would be $O(h^2+ |\Delta x|^2/h)$, which leads to order $2$ in $h$ and $4/3$ in $|\Delta x|$.
However, no improvement of the upper bound is guaranteed due again to the control approximation, which, as explained in \cite{JakoPicaReisi18}, can be improved only if $v_h$ is smooth too, which is usually not the case even if $v$ is.

\subsection{Higher order interpolation}
\label{sec:interp}

A remaining bottleneck is the accumulated interpolation error $|\Delta x|^2/h$, which is dictated by the need for (multi-)linear interpolation to ensure  the monotonicity of the scheme.
Some recent results (see \cite{reisinger2016piecewise}) indicate that monotonicity of the interpolation step is not needed to ensure convergence of the scheme, as long as the interpolation is ``limited'' to avoid overshoots. An interesting example is 
the monotonicity preserving cubic interpolation (see \cite{fritsch1980monotone}, and \cite[Section 6]{DJ12} for an application to semi-Lagrangian schemes) which preserves the monotonicity of the input data in intervals where the data are monotone, and is of high order if the data are monotone overall. In special cases where the monotonicity of the value function is known a priori (such as typical utility maximisation problems in finance), this may lead to a practical improvement of the order,
{  as evidenced in our numerical tests}, although a theoretical proof of the higher order seems difficult.

{ 
\section{Numerical tests}\label{sec:num}

In what follows, as application of our method and to test its numerical properties,
we consider a problem from mathematical finance which consists in pricing a European option under a Black-Scholes model with unequal lending and borrowing rates.
This model, originally proposed in \cite{bergman1995option} and analysed analytically  in \cite{amadori2003nonlinear} as special case of the framework studied therein, has frequently been used as a test case for numerical methods in the literature
(see, e.g., \cite{forsyth2007numerical, witte2011penalty} for the solution of HJB PDEs by discretisation and penalisation, respectively, \cite{gobet2005regression,bender2007forward,bender2017primal}
for regression-based BSDE methods, and \cite{weinan2017deep} for a deep learning method for the PDE solution).

The relaxation of the assumption of a single funding rate has recently attracted renewed attention in the financial industry in the context of collateralization (see, e.g., \cite{mercurio2015bergman}).

The market frictions introduce a nonlinearity of the pricing rule with respect to the payoff and an asymmetry between long and short positions in the option.
We focus here on the latter.
Appendix A.1 in \cite{forsyth2007numerical} gives a derivation -- by a hedging argument -- of
the following HJB equation for the value $u(s,t)$ of the option at time $t$ given a value of the underlying asset of $s$,
$$
\partial_t u + \frac12 \sigma^2 s^2 \partial_{ss} u +\underset{q\in \{r_b, r_l\}}\sup \Big\{ q (s \partial_s u - u)\Big\}=0,
$$
considered with a terminal condition $u(T,s) = \psi(s)$, where $\psi$ is the payoff function of the option at maturity $T$,
$r_b$ and $r_l$ are the borrowing and lending rates, respectively, and $\sigma$ the volatility.
With the change of variable $x=\log(s)$ one can pass to the function $v(t,x):= u(t,{e}^x)$ as solution of
\begin{eqnarray}
\label{hjb-rates-x}
\partial_t v + \frac12 \sigma^2  \partial_{xx} v -\frac12 \sigma^2 \partial_x v +\underset{q\in \{r_b, r_l\}}\sup \Big\{ q ( \partial_x v - v)\Big\}=0.
\end{eqnarray}
We note that this PDE is semi-linear, with a nonlinearity in the first and zero order terms. In the situation of a fixed, globally constant optimiser $q$, the equation has constant coefficients.

The option price can be interpreted as the value function of a control problem as follows,
$$
v(t,x) = \sup_{\zeta \in \mathcal Q} \mathbb E_{t,x}\left[ {e}^{-\int_t^T \zeta_r \, \mathrm d r } \psi(\exp(X^\zeta_T))\right],
$$
with the set $\mathcal Q$ of progressively measurable processes with values in $\{r_b, r_l\}$ 
and $X^\zeta_T= x +\int^T_t (\zeta_r-\sigma^2/2) \, \mathrm d r  + \int^T_t \sigma \, \mathrm d B_r$.

As the volatility is constant, independent of the control process, this is an ideal test for our method as it singles out the error from the Markov chain approximation and allows us to assess the improvement achieved by a higher order quadrature rule. The optimal feedback control is piecewise constant as a function of both time and space, with a small number of jump points for practically relevant payoffs.


This fits into the previous framework with the minor extension of a controlled discount factor. For the value function $v_h$ of the problem with piecewise constant control processes we have the following 
DPP:
\begin{equation}\label{eq:DPP_v_rate}
v_h(t,x)= \underset{q\in \{r_b,r_l\}}\sup \;\mathbb E_{t,x}\left[{e}^{-q h} v(t+h,X^{q}_{t+h})\right],
\end{equation}
where $X^{t,x,q}_{t+h}= x +  (q - \frac{\sigma^2}{2})h + \sigma (B_{t+h} - B_t)$.
Therefore, adapting to the present case the definitions given in Section \ref{sect:disc_scheme}, in particular \eqref{eq:SL_V_fully}, we can define a fully discrete scheme by 
$$
V(t_n,x_m) = \sup_{q\in \{r_b,r_l\}} \sum^{{M}}_{i=1} e^{-q h}  { \lambda_i}\; \mathcal I[V]\Big(t_{n+1},x_m + (q- \frac{\sigma^2}{2} ) h+\sqrt{h}\sigma \; {\xi_i}\Big)
$$
with $n=0,\ldots , N-1$ and $m = 0,\ldots, J$.

We consider as terminal conditions a  a call payoff
$$
\psi(s) = (s-K)^+
$$
and a so-called butterfly payoff 
$$
\psi(s) = 0.25 \left( (s - K1)^+ - 2 (s-0.5(K1 +K2))^+ + (s-K2)^+ \right).
$$
The parameters used are given in Table \ref{tab:data}.
\begin{table}[!hbtp]
\centering
\begin{tabular}{|c|c|c|c|c|c|c|}
\hline 
$r_l$ & $r_b$  & $\sigma$ &  $K_1$ & $K_2$ & $K$ & $T$\\
\hline
$0.1$ & $0.15$ & $0.4$  & $100$ & $300$ & $100$ & $1$\\
\hline
\end{tabular}
\caption{Parameters used in numerical experiments (same as in \cite{witte2011penalty} with $r_f=0$).}
\label{tab:data}
\end{table}

The call payoff and numerical solution for the value function are given in Figure \ref{fig:call}, left. The butterfly payoff is shown in Figure \ref{fig:call}, right, while the
numerical approximation to the value function is shown in Figure \ref{fig:butt}, left, blue solid curve. 
 {
In the call case, the optimal control is constant at $q={r_b}$.
For the butterfly, in addition to ${(r_l,r_b)}=(0.1,0.15)$, as used elsewhere throughout the paper, we also plot the value functions for
$(r_l,r_b)=(0.1,0.1)$  (dash-dotted magenta curve) and 
$(r_l,r_b)=(0.15,0.15)$  (dashed black curve), which are seen to be strictly smaller. This demonstrates that the constant strategies are sub-optimal.}

The focus of our tests is to establish whether
 a higher order version of the semi-Lagrangian scheme with $M>2$, in particular the case $M=4$, leads to better accuracy and higher convergence rate than the standard case $M=2$.
To this end, we compute numerical approximations with an increasing number of timesteps, $N$. 

We note that the Euler scheme for fixed control is exact in this setting, and the time discretisation error therefore of order $h^3$ for $M=4$, provided smooth enough solutions. 
In order to balance this term and the interpolation error of order $(\Delta x)^2/h$, we choose $\Delta x \sim h^2$ (an `inverse CFL condition') or $J\sim N^2$. In the case of $M=2$, this makes the spatial error negligible compared to the time stepping error of order $h$.

In Table \ref{tab:call}, we list the error evaluated in the maximum norm over a suitable spatial interval 
and the resulting estimated convergence order, for the call payoff (observe that in this case the problem becomes linear with $q={r_b}$ and an exact solution is available). 
In this case, as the optimal control is constant, this allows us to verify the achievable order in the simplest case of a linear problem with constant coefficients. Indeed, the order is approximately 1 for $M=2$ and 3 for $M=4$.

\begin{figure}
\includegraphics[width=0.495\textwidth]{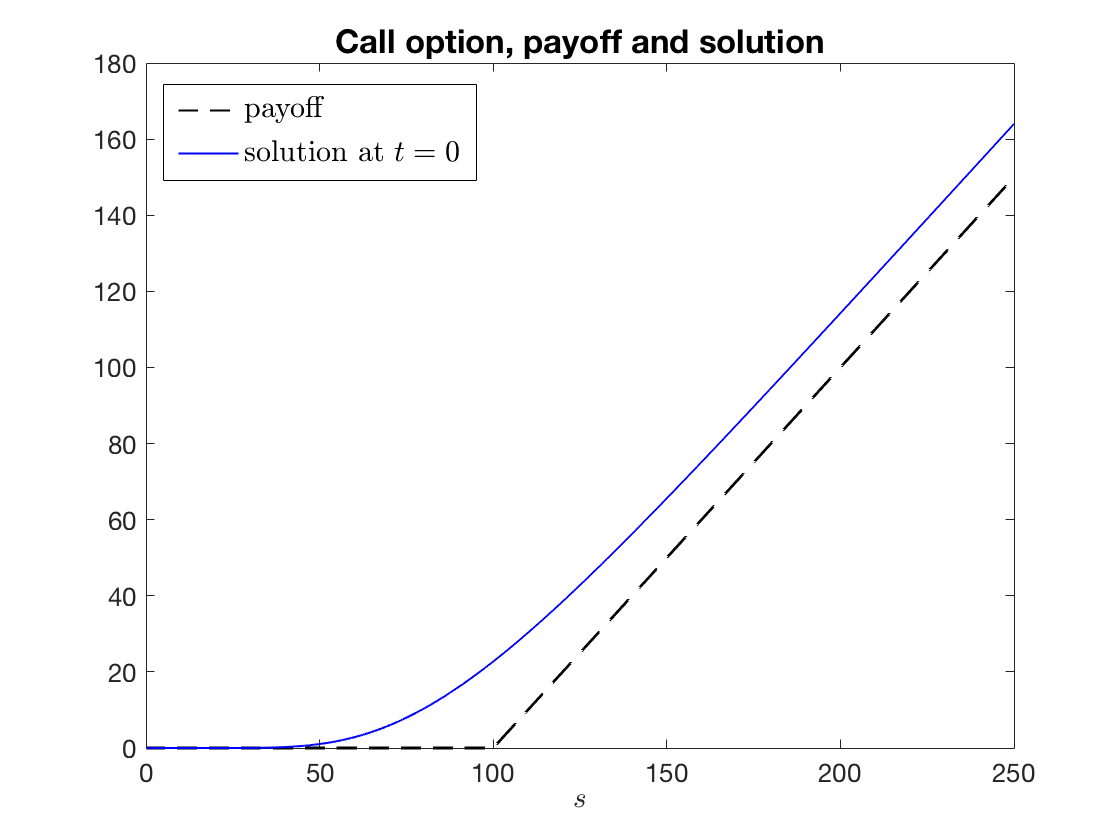} \hfill
\includegraphics[width=0.495\textwidth]{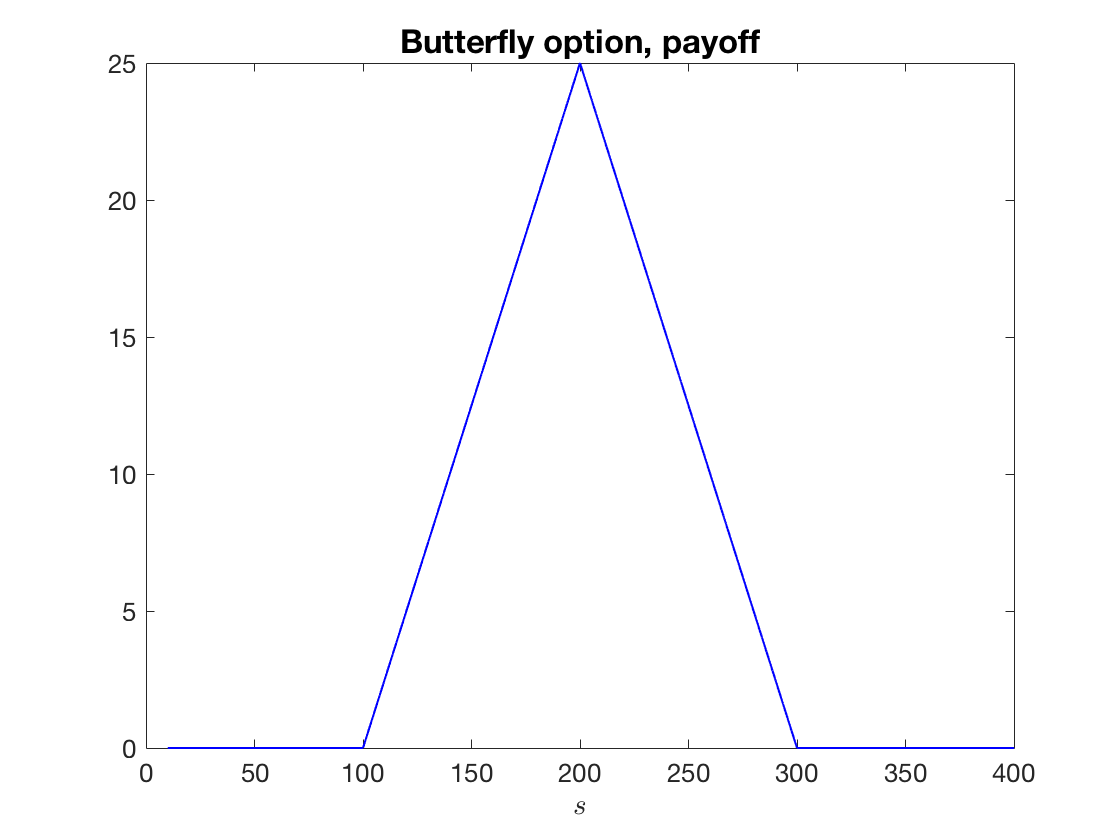}
 \caption{Left: call option payoff (dashed black line) and solution at $t=0$ (solid blue line). Right: butterfly option payoff. }
 \label{fig:call}
\end{figure}

\begin{table}[h]
\begin{tabular}{|c|ccc|ccc|}
\hline
&\multicolumn{3}{c|}{$M=2$} & \multicolumn{3}{c|}{$M=4$}\\
\hline
$k$ & error & order  & CPU (s) & error & order & CPU (s)  \\
\hline
1 & 2.25E-01 & - & 0.02 &   5.99E-01 & - & 0.04\\
   2 & 5.97E-02 &  1.91 &  0.03 & 6.94E-02 &  3.11 & 0.08\\ 
    3 & 2.64E-02 &  1.18  &  0.13 & 9.80E-03 &  2.82 & 0.47 \\ 
    4 & 1.20E-02 &  1.13  &  0.92 & 2.63E-03 &  1.90 & 1.54\\ 
    5 & 6.06E-03 &  0.99  & 3.85 & 3.48E-04 &  2.92 & 8.10\\ 
    6 & 3.28E-03 &  0.88  & 26.00 & 3.00E-05 &  3.54  & 59.58\\ 
    7 & 1.75E-03 &  0.91  & 276.01 & 2.91E-06 &  3.36 & 613.85\\ 
   8 & 9.02E-04 &  0.95  & 2382.37 & 4.40E-07 &  2.73 & 4834.02\\ 
\hline 
  \end{tabular}
\caption{\label{tab:call} Call option. Refinements $N = 2^4\cdot 2^{k}$ and $J = N^2/4 $ for $k=1,\ldots,8$. Local error computed for $s \in [70, 90]$. }
\end{table}

\begin{figure}
\includegraphics[width=0.495\textwidth]{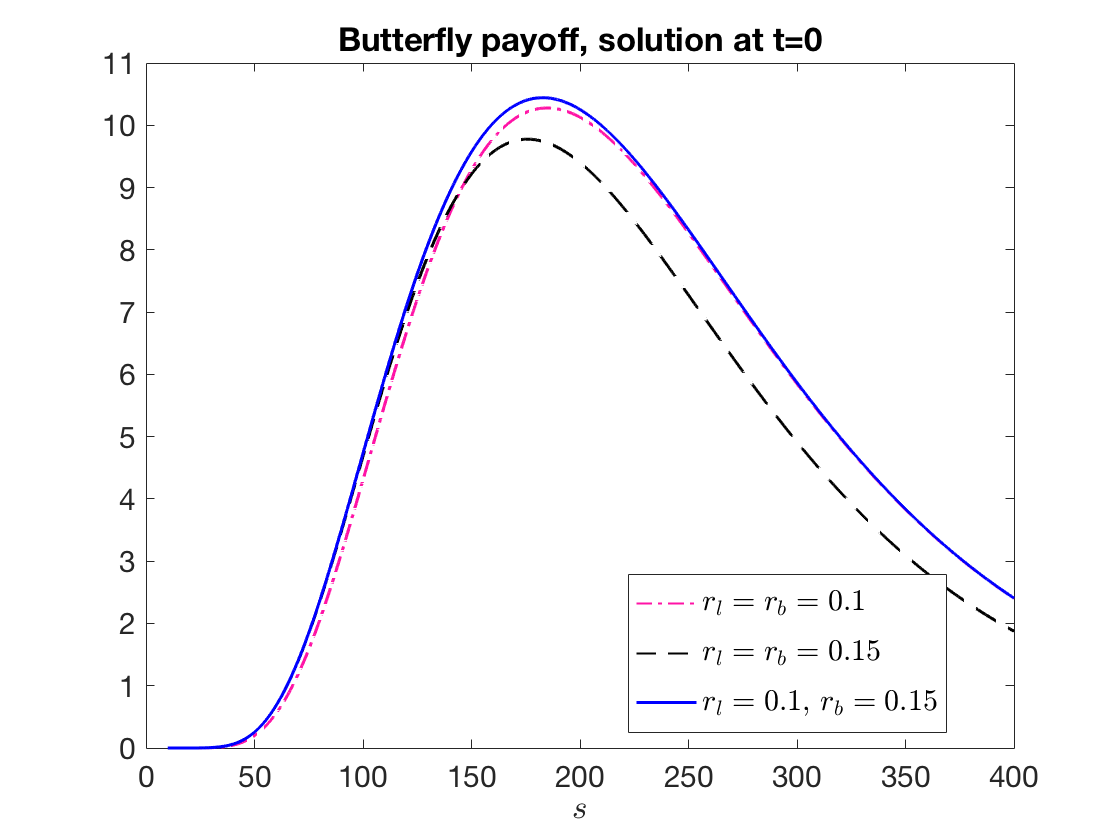}\hfill
\includegraphics[width=0.495\textwidth]{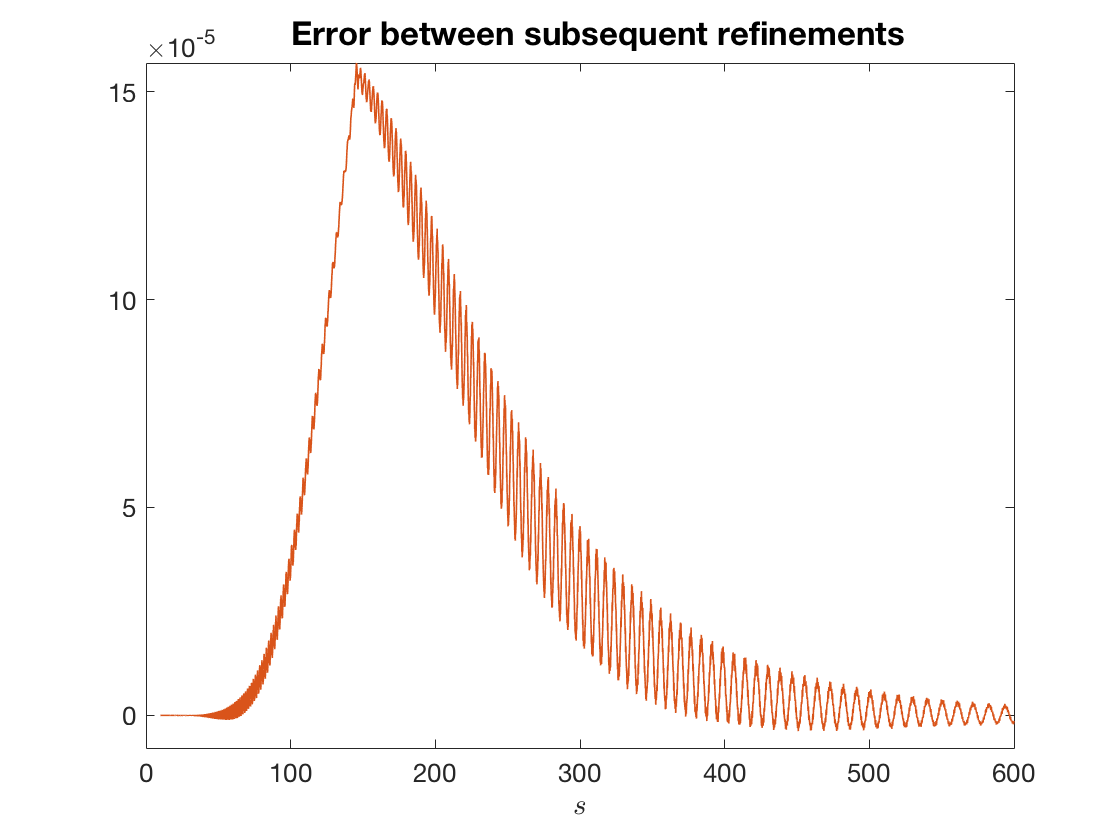}
 \caption{Butterfly option. Left: solution at $t=0$ (solid blue curve), compared with the solution of the linear problems $r_l=r_b=0.1$ (dash-dotted magenta curve) and $r_l=r_b=0.15$ (dashed black curve). Right: difference between solutions at $t=0$ computed with two subsequent mesh refinements ($k=7$ and $k=8$, corresponding to last line in Table \ref{tab:butt}).}
 \label{fig:butt}
\end{figure}

The corresponding results for the butterfly payoff are given in Table \ref{tab:butt} (in this case, we do not have an exact  solution and the error is considered as the difference between numerical solutions for subsequent mesh refinements).
A typical shape of the error as a function of $s$ (at $t=0$) is shown in Figure \ref{fig:butt}, right.
Here, the optimal control is piecewise constant in space, switching from $r_b$ to $r_l$ at a time-dependent value of $s$, between 150 and 200 in this case.
This results in a `kink' in the controlled term in \eqref{hjb-rates-x}, and therefore the best regularity we can expect is that $u_{xx}$ is Lipschitz at this point, and smooth everywhere else.
We therefore show separately the maximum error in two different intervals, namely $[30, 70]$, sufficiently far from the non-smooth point, and $[130, 170]$, containing the non-smooth point.

\begin{table}[h]
\begin{tabular}{|c|cc|cc|c|cc|cc|c|}
\hline
&\multicolumn{5}{c|}{$M=2$} & \multicolumn{5}{c|}{$M=4$}\\
\hline
& \multicolumn{2}{|c|}{$s\in[30, 70]$} & \multicolumn{2}{|c|}{$s\in[130, 170]$} &   \multirow{2}{*}{CPU(s)} & \multicolumn{2}{|c|}{$s\in[30, 70]$} & \multicolumn{2}{|c|}{$s\in[130, 170]$} & \multirow{2}{*}{CPU(s)} \\
\cline{1-5}\cline{7-10}
$k$ & error & order  & error & order  &  & error & order &  error & order & \\
\hline
    1  & 1.39E\,00 &  - & 1.81E\,00 &  - & 0.01 & 1.11E\,00 &  -  & 1.36E\,00 &  - & 0.01 \\ 
    2  & 7.65E-02 &  4.18 & 1.81E-01 &  3.33  & 0.01 & 1.69E-01 &  2.72 & 3.74E-01 &  1.86 & 0.02 \\ 
    3  & 1.65E-02 &  2.21 & 5.13E-02 &  1.82 & 0.02 &   2.40E-02 &  2.81 & 6.39E-02 &  2.55 & 0.05 \\ 
    4  & 1.02E-02 &  0.70  &  3.75E-02 &  0.45  &  0.10 &  3.63E-03 &  2.73 & 1.53E-02 &  2.06 & 0.28  \\ 
    5 & 4.65E-03 &  1.13   & 1.44E-02 &  1.38   &  0.80 & 1.35E-03 &  1.42 & 5.39E-03 &  1.51 & 1.10  \\ 
    6 & 3.04E-03 &  0.61  & 1.13E-02 &  0.35  & 3.98 & 4.03E-04 &  1.75  & 1.44E-03 &  1.91  &  7.12 \\ 
    7 &  1.89E-03 &  0.68 & 4.99E-03 &  1.18 & 28.54 & 5.48E-05 &  2.88 & 3.51E-04 &  2.03 &  54.76\\
    8 & 8.46E-04 &  1.16 & 2.32E-03 &  1.10 & 302.29 &  5.77E-06 &  3.25  & 1.57E-04 &  1.16 & 530.14\\
\hline 
  \end{tabular}
\caption{\label{tab:butt} Butterfly option. Refinements $N = 2^3\cdot 2^{k}$ and $J = N^2/8 $ for $k=1,\ldots,8$. }
\end{table}
The order of the scheme with $M=2$ is still 1 in both cases.
For $M=4$, the order in the first interval is still around 3, whereas it is reduced to around 2 close to the non-smooth point.

  Finally, in Table \ref{tab:butt_chip}, we report results where we replace the piecewise linear spatial interpolation with the piecewise cubic, and piecewise monotone interpolation
  defined in \cite{fritsch1980monotone}, as discussed in Section \ref{sec:interp}. Conjecturing an interpolation error of order $\Delta x^4$ per timestep, the most efficient refinement is achieved by balancing the time accumulated error $\Delta x^4/h$ with $h^3$ (for $M=4$) and we hence choose $\Delta x \sim h$, or $J \sim N$.
  \begin{table}[h]
\begin{tabular}{|c|cc|cc|c|cc|cc|c|}
\hline
&\multicolumn{5}{c|}{$M=2$} & \multicolumn{5}{c|}{$M=4$}\\
\hline
& \multicolumn{2}{|c|}{$s\in[30, 70]$} & \multicolumn{2}{|c|}{$s\in[130, 170]$} &   \multirow{2}{*}{CPU(s)} & \multicolumn{2}{|c|}{$s\in[30, 70]$} & \multicolumn{2}{|c|}{$s\in[130, 170]$} & \multirow{2}{*}{CPU(s)} \\
\cline{1-5}\cline{7-10}
$k$ & error & order  & error & order  &  & error & order &  error & order & \\
\hline
    1  &1.05E-01 &  - & 5.16E-01 &  -               &  0.01 & 4.39E-02 &  - &   2.06E-01  &  -  & 0.03 \\ 
    2  & 5.78E-02 &  0.86  & 2.00E-01    &  1.37 & 0.02 & 2.33E-02 &  0.91 &    1.18E-01 &  0.80   & 0.06 \\ 
    3  & 1.58E-02 &  1.87   &   7.91E-02   &  1.34  & 0.04 &  1.03E-02 &  1.18 &  4.23E-02 &  1.48 & 0.18 \\ 
    4  & 1.14E-02 &  0.47    &    3.64E-02  &   1.12     &  0.09 &  4.22E-03 &  1.29 &  1.62E-02 &  1.39  & 0.73  \\ 
    5 &  5.28E-03 &  1.11  	&  1.54E-02 &     1.24   &  0.38 & 1.45E-03 &  1.54 & 5.44E-03 &  1.57   &  2.79  \\ 
    6 & 3.07E-03 &  0.78     &   1.18E-02  &   0.38   & 6.66 & 4.13E-04 &  1.81  & 1.36E-03 &  2.00 &  13.22 \\ 
    7 &  1.86E-03 &  0.72    &  5.01E-03   &    1.24     & 27.54 & 5.51E-05 &  2.91 & 3.47E-04 &  1.98 &   53.53\\
    8 & 8.34E-04 &  1.16     &    2.30E-03   &    1.12   & 111.57 &  6.17E-06 &  3.16 & 1.56E-04 &  1.15  &  218.77\\
\hline 
  \end{tabular}
\caption{\label{tab:butt_chip} Butterfly option. Refinements $N = 2^3\cdot 2^{k}$ and $J = N/16 $ for $k=1,\ldots,8$, using ``monotonicity preserving'' third order interpolation.}
\end{table}
  The results are very similar to those in Table \ref{tab:butt}, but obtained for a significantly smaller number of spatial nodes. Due to the higher computational cost of the interpolation, though, this only results in savings for the highest refinement levels reported here.

}

\section{Conclusions and perspectives}
\label{sec:concl}

This paper analyses numerical schemes for HJB equations based on a discrete time approximation of the optimal control problem. Using purely probabilistic arguments and under very general assumptions, in Section \ref{sec:lower} we give a bound for the solution generated by such an approximation. The error bound obtained in this way allows us to improve one side of previous results from the literature. 
In ongoing work \cite{picarelli2019duality}, we are investigating the use of duality to obtain symmetric bounds.

{ 
The theoretically obtained convergence orders, although sharper than previously known results, are still not sharp in applications, where the solution is often at least piecewise more regular than is assumed for the analysis.
The new technique of splitting the total error into contributions from different approximation stages, however, allows us to understand why our higher order version outperforms the standard semi-Lagrangian scheme in numerical tests, despite the fact that the consistency order used in the traditional analysis is identically 1 for all schemes.
}
\\
\\
{\bf Acknowledgements}: {The first author acknowledges Olivier Bokanowski for a preliminary discussion on the subject.}

\appendix

\section{Bounds for the  Euler-Maruyama approximation}
 {
We consider the Euler-Maruyama approximation given by \eqref{eq:def_M} for $\alpha\equiv(a_0,\ldots, a_{N-1})\in \mathcal A_h$. This leads to the following expression for $\widetilde X^{t_n,x,\alpha}_{t_k}$ for $k=n,\ldots,N$:
$$
\widetilde X^{t_n,x,\alpha}_{t_k} = x + \sum^{k-1}_{i=n} \int^{t_{i+1}}_{t_i} \mu (t_i,\widetilde X^{t_n,x,\alpha}_{t_i},a_i) \, \mathrm{d}s + \int^{t_{i+1}}_{t_i} \sigma (t_i,\widetilde X^{t_n,x,\alpha}_{t_i},a_i) \, \mathrm{d} B_s.
$$
Moreover, by the very definition of $X^{t_n,x,\alpha}_{\cdot}$:
$$
 X^{t_n,x,\alpha}_{t_k} = x + \sum^{k-1}_{i=n} \int^{t_{i+1}}_{t_i} \mu (s,X^{t_n,x,\alpha}_{s},a_i) \, \mathrm{d}s + \int^{t_{i+1}}_{t_i} \sigma (s,X^{t_n,x,\alpha}_{s},a_i) \, \mathrm{d} B_s.
$$
Hereafter, we do not keep track of individual constants and denote by $C$ any nonnegative constant depending only on $T$ and  $C_0$ in assumption (H2) (an explicit computation of the constants involved can be found in  \cite[Appendix A.1]{picarelli2019duality}).
Using the Cauchy-Schwartz inequality and It{\^o} isometry, and (H2), one can easily show that
\begin{align*}
\mathbb E_{t_n,x}\left[ |\widetilde X^{\alpha}_{t_k} - X^{\alpha}_{t_k}|^2\right] 
& \leq C h \sum^{k-1}_{i=n}  \mathbb E_{t_n,x}\left[ | \widetilde X^{\alpha}_{t_i} - X^{\alpha}_{t_i} |^2 + h+ \underset{s\in [t_i,t_{i+1}]}\sup\left| X^{\alpha}_{s} - X^{\alpha}_{t_i}\right|^2  \right].
\end{align*} 
We now estimate the last term on the right-hand side. For any $\alpha\in \mathcal A$, by the Cauchy-Schwartz and Doob maximal inequalities, one has 
\begin{align*}
\mathbb E_{t_n,x}\left[\underset{s\in [t_i,t_{i+1}]}\sup\left| X^{\alpha}_{s} - X^{\alpha}_{t_i}\right|^2  \right]\leq
 C\, \mathbb E_{t_n,x}\left[ h \int^{t_{i+1}}_{t_i} | b(r,X^\alpha_r,\alpha_r) |^2 \mathrm d r    
 +  \int^{t_{i+1}}_{t_i} | \sigma(r,X^\alpha_r,\alpha_r) |^2 \mathrm d r\right]
\end{align*}
for some constant $C\geq 0$ independent of $\alpha$. 
Then, thanks to the linear growth of the coefficients $b$ and $\sigma$ due to assumption (H2), one obtains
\begin{align*}
\mathbb E_{t_n,x}\left[\underset{s\in [t_i,t_{i+1}]}\sup\left| X^{\alpha}_{s} - X^{\alpha}_{t_i}\right|^2  \right]\leq
 C\, \mathbb E_{t_n,x}\left[ h + \int^{t_{i+1}}_{t_i}  | X^\alpha_r|^2 \mathrm d r   \right].
\end{align*}
Recalling that by classical estimates on the process $X^{t_n,x}_\cdot$ {  (see for instance \cite[Theorem 3.1]{Touzi_book})  }one has 
\begin{align*}
{\mathbb E_{t_n,x}\Big[ \underset{s\in [t_i,t_{i+1}]}\sup \left| X^{\alpha}_{s}\right|^2 \Big]\leq C\left(1+ |x|^2\right)},
\end{align*}
we can put these estimates together to obtain
\begin{align*}
\mathbb E_{t_n,x}\left[ |\widetilde  X^{\alpha}_{t_k} - X^{\alpha}_{t_k}|^2\right] & \leq C h \sum^{k-1}_{i=n} \mathbb E_{t_n,x}\left[ | \widetilde X^{\alpha}_{t_i} - X^{\alpha}_{t_i}|^2\right]
 +  C h (1+|x|^2).
\end{align*}
Then, using the discrete version of Gronwall's inequality, one can conclude that for any $k=n,\ldots, N$,
\begin{align*}
\mathbb E_{t_n,x}\left[ |\widetilde  X^{\alpha}_{t_k} - X^{\alpha}_{t_k}|^2\right]  \leq  C h  (1+|x|^2)
\end{align*}}
for some constant $C$ independent of $h$ and $\alpha\in \mathcal A_h$.
{ The result of Proposition \ref{prop:strong_conv} then follows from H\"older's inequality.}
\bibliography{biblio.bib}
\bibliographystyle{plain}

\end{document}